\documentclass[oneside,english]{amsart}
\usepackage[T1]{fontenc}
\usepackage[latin9]{inputenc}
\usepackage{amsthm}
\usepackage{amstext}
\usepackage{amssymb}
\usepackage{esint}
\usepackage{graphicx}
\usepackage{enumerate}
\usepackage{amscd}

\makeatletter
\newtheorem{theorem}{Theorem}[section]
\newtheorem{lemma}[theorem]{Lemma}

\numberwithin{figure}{section}
\theoremstyle{definition}
\newtheorem{definition}[theorem]{Definition}
\newtheorem{example}[theorem]{Example}

\theoremstyle{remark}
\newtheorem{remark}[theorem]{Remark}

\numberwithin{equation}{section}
\makeatother

\usepackage{color}

\usepackage{babel}

	\DeclareMathOperator*{\esssup}{ess\,sup}

\begin{document}

\title[Spectral estimates of the Dirichlet-Laplace operator]{Spectral estimates of the Dirichlet-Laplace operator
in conformal regular domains}

\author{Ivan Kolesnikov and  Valerii Pchelintsev}
\begin{abstract}
In this paper we consider conformal spectral estimates of the Dirichlet-Laplace operator in conformal regular domains $\Omega \subset \mathbb R^2$. This study is based on the geometric theory of composition operators on Sobolev spaces that permits
us to estimate constants of the Poincar\'e-Sobolev inequalities. On this base we obtain lower estimates of the first eigenvalue of the Dirichlet-Laplace operator in a class of conformal regular domains. As a consequence we obtain conformal estimates of the ground state energy of quantum billiards.
\end{abstract}
\maketitle

\section{Introduction}
This paper is devoted to study conformal spectral estimates of the Laplace operator with the Dirichlet boundary condition:
\begin{equation}\label{Dir}
-\Delta u=\lambda u\text{ in }\Omega, \quad u=0\text{ on }\partial \Omega,
\end{equation}
in bounded simply connected domains $\Omega \subset \mathbb R^2$.

We consider the Dirichlet eigenvalue problem~\eqref{Dir} in the weak formulation.
Recall that a function $u\in W_0^{1,2}(\Omega)$ is a weak solution of the spectral problem for the Laplace
operator with the Dirichlet boundary condition if
\begin{equation}\label{WFEP}
\iint\limits_{\Omega} \left\langle \nabla u(x,y), \nabla v(x,y) \right\rangle dxdy\\
= \lambda \iint\limits_{\Omega} u(x,y)v(x,y)~dxdy
\end{equation}
for all $v\in W_0^{1,2}(\Omega)$.

By the Sobolev embedding theorems \cite{M} it follows that the Dirichlet problem admit a sequence of nonnegative eigenvalues, which we denote by
\[
0<\lambda_1(\Omega) \leq \lambda_2(\Omega) \leq \ldots \leq \lambda_n(\Omega) \leq \ldots\,.
\]
In this sequence each eigenvalue is counted according with its multiplicity (see \cite{Henr,KS}). By the minimax principle, the first eigenvalue $\lambda_1(\Omega)$ can be characterized as
\[
\lambda_1(\Omega)= \inf_{u \in W_0^{1,2}(\Omega) \setminus \{0\}} \frac{\iint\limits_{\Omega} |\nabla u(x,y)|^2~dxdy}{\iint\limits_{\Omega} |u(x,y)|^2~dxdy}\,.
\]

The exact calculations of $\lambda_1(\Omega)$ in arbitrary domains is a long-standing complicated problem. In fact,
explicit values of $\lambda_1(\Omega)$ are known only for several particular domains. Among them, one finds (see \cite{KS}): the rectangle, the disc, the equilateral triangle. Thus, it is important to obtain sharp estimates on $\lambda_1(\Omega)$, which depend on an intrinsic geometry of domains. The most famous instance of such an estimate is the so-called
Rayleigh-Faber-Krahn inequality \cite{F23,Kh25}. This inequality states that the disc minimizes the first Dirichlet eigenvalue of the Laplace operator among all planar domains of the same area:
\begin{equation}\label{RFK}
\lambda_1(\Omega)\geq \lambda_1(\Omega^{\ast})=\frac{{j_{0,1}^2}}{R^2_{\ast}},
\end{equation}
where $j_{0,1} \approx 2.4048$ is the first positive zero of the Bessel function $J_0$ and $\Omega^{\ast}$ is a disc with radius $R_{\ast}$ of the same area as $\Omega$. Note that the Rayleigh-Faber-Krahn inequality was refined using the capacity method, see $\S 4.7$ in \cite{M}.

The estimate~\eqref{RFK} is not optimal, especially it's very rough for narrow and elongated domains. We easily verify it considering the rectangle $Q:=(0,a)\times(0,b)$. For the rectangle $Q$ the explicit value is \cite{Henr} 
\[
\lambda_1(Q)=\frac{\pi^2}{a^2}+\frac{\pi^2}{b^2}
\]
and the estimate~\eqref{RFK} takes the form
\[
\lambda_1(Q)\geq \frac{\pi{j_{0,1}^2}}{ab},
\]
that for $a \gg b$ is a rough estimate.

In simply connected planar domains another lower bound for the first Dirichlet eigenvalue of the Laplace operator was obtained by Makai \cite{M65}:
\begin{equation}\label{M-H}
\lambda_1(\Omega)\geq \frac{\gamma}{\rho^2},
\end{equation}
where $\gamma =1/4$ and $\rho$ is the radius of the largest disc inscribed in $\Omega$. For convex domains, this lower bound with $\gamma = \pi^2/4$ was obtained by Hersch \cite{H60}.

In the present paper we show that integrability of conformal mappings's derivative with the exponent greater than two permit us to obtain lower estimates of the first eigenvalue $\lambda_1(\Omega)$ in terms of Sobolev norms of conformal mappings of $\Omega'$ onto $\Omega$, where $\Omega'$ is a bounded simply connected domain. So, we can conclude that $\lambda_1(\Omega)$ depends on the conformal geometry of $\Omega$.

\vskip 0.2cm
\noindent
\textbf{Theorem A.}\label{thA}
\textit{Let $\Omega$ be a conformal $\infty$-regular domain about a domain $\Omega'$. Then
\begin{equation}\label{Our}
\lambda_1(\Omega) \geq \frac{\lambda_1(\Omega')}{\big\|\varphi' \mid L^{\infty}(\Omega')\big\|^2},
\end{equation}
where $\varphi:\Omega' \to \Omega$ is the Riemann conformal mapping of $\Omega'$ onto $\Omega$.
}

\vskip 0.2cm

Note that the estimate~\eqref{Our} in some subclass of domains is better than the classical estimates~\eqref{RFK} and~\eqref{M-H}.
As an example consider a lower estimate of the first eigenvalue of Laplace operator in the non-convex domain
\[
\Omega:= \left\{w:1<|w|<e^d,\,\, 0<\arg w<\pi \right\},
\]
which is the image of the rectangle $Q:=(0,d) \times (0,\pi)$ under the conformal mapping $\varphi(z)=e^z.$ Since $\lambda_1(Q)=1+\pi^2/d^2$ and $\big\|\varphi' \mid L^{\infty}(Q)\big\|^2 \leq e^{2d}$, by Theorem~A we have
\begin{equation*}
\lambda_1(\Omega_e)\geq \frac{\pi^2+d^2}{d^2e^{2d}}.
\end{equation*}
If $0<d\leq \ln 2$ then this estimate is better than the estimates based on the Rayleigh-Faber-Krahn and Makai inequalities (see Example 4.4). More examples will be given in Section 4.

Let us remind some concepts and facts related to the Theorem A.
\begin{definition}
Let $\Omega,\Omega' \subset \mathbb R^2$ be simply connected domains. Then a domain $\Omega$ is called a conformal $\alpha$-regular domain about a domain $\Omega'$ if
\[
\iint\limits_{\Omega'} |\varphi'(u,v)|^{\alpha}~dudv<\infty
\]
for some $\alpha>2$, where $\varphi:\Omega' \to \Omega$ is a conformal mapping \cite{BGU15}.
\end{definition}

This concept does not depend on choice of a conformal mapping $\varphi:\Omega' \to \Omega$ and can be reformulated in terms of the conformal radii. Namely
\[
\iint\limits_{\Omega'} |\varphi'(u,v)|^{\alpha}~dudv=
\iint\limits_{\Omega'} \left( \frac{R_{\Omega}(\varphi(u,v))}{R_{\Omega'}(u,v)}\right)^{\alpha}dudv,
\]
where $R_{\Omega'}(u,v)$ and $R_{\Omega}(\varphi(u,v))$ are conformal radii of $\Omega'$ and $\Omega$ \cite{Avkh,BM}.
The domain $\Omega$ is a conformal regular domain if it is an $\alpha$-regular domain for some $\alpha >2$. This class of domains includes, in particular, domains with a Lipschitz boundary and fractal domains of the snowflake type \cite{GPU17_2}. The Hausdorff dimension of the conformal regular domain's boundary can be any number in [1, 2) \cite{HKN}.

\begin{remark}
Theorem~A can be reformulated in terms of the conformal radii
$R_{\Omega}(\varphi(x,y))$ of the domain $\Omega$.
\end{remark}

In the case of convex domains the estimate~\eqref{Our} can be refined in terms of the outer and inner radii of the domain, and the curvature radius of its boundary (see Theorem~\ref{convex}).

Using Theorem~A and the domain monotonicity property ($\Omega \subseteq \Omega'$ implies $\lambda_1(\Omega)-\lambda_1(\Omega')\geq 0$) we obtain an estimate for variation of the first eigenvalue of the Laplace operator with the Dirichlet boundary condition in conformal regular domains (see Theorem~\ref{var}).

The suggested method is based on the geometrical theory of composition operators on Sobolev spaces \cite{GG94,GU09,U93,VU02} in the special case of operators generated by conformal mappings. Applications of conformal composition operators were given in \cite{BGU15,BGU16,GPU17_2,GPU2018,GPU18,GU16}, where they consider the Dirichlet and Neumann eigenvalue problems for conformal regular domains.

We also note that the method based on conformal composition operators allows one to estimate the ground state energy of quantum billiards with boundaries from a large class of domains. The conformal mapping theory is widely used in study quantum billiards \cite{A10,R84}. Quantum billiards have been studied extensively in recent years (see, for example, \cite{A10,Kh20,LHK,CGP}). These systems describe the motion of a point particle undergoing perfectly elastic collisions in a bounded domain with the Dirichlet boundary condition.
Billiards have been applied in several areas of physics to model quite diverse real world systems. Examples include ray-optics \cite{KM}, lasers \cite{S10}, acoustics \cite{KNK}.

The paper is organized as follows: In Section 2 we give basic facts about composition operators on Sobolev spaces and the conformal mapping theory. In Section 3 we prove the Poincar\'e-Sobolev inequality for conformal
regular domains. In Section 4 we apply results of Section 3 to lower
estimates of the first Dirichlet eigenvalue of the Laplace operator in
the conformal regular domains. We obtain an estimate for variation of the first Dirichlet eigenvalue of the Laplace operator under conformal deformations of the domain. Also we refine these estimates and estimates of the spectral gap for convex domains.

\section{Sobolev spaces and conformal mappings}
Let $E \subset \mathbb R^2$ be a measurable set and $h:E \to \mathbb R$ be a positive a.e. locally integrable function i.e. a weight. The weighted Lebesgue space $L^p(E,h)$, $1\leq p<\infty$,
is the space of all locally integrable functions with the finite norm
$$
\|f\,|\,L^{p}(E,h)\|= \left(\iint\limits_E|f(x,y)|^ph(x,y)\,dxdy \right)^{\frac{1}{p}}< \infty.
$$
In the case $h=1$ this weighted Lebesgue space coincides with the classical Lebesgue space $L^{p}(\Omega)$.

The Sobolev space $W^{1,p}(\Omega)$, $1\leq p< \infty$, is defined
as the normed space of all locally integrable weakly differentiable functions
$f:\Omega\to\mathbb{R}$ endowed with the following norm:
\[
\|f\mid W^{1,p}(\Omega)\|=\|f\,|\,L^{p}(\Omega)\|+\|\nabla f\mid L^{p}(\Omega)\|,
\]
where $\nabla f$ is the weak gradient of the function $f$.
Recall that the Sobolev space $W^{1,p}(\Omega)$ coincides with the closure of the space of smooth functions $C^{\infty}(\Omega)$ in the norm of $W^{1,p}(\Omega)$.

The seminormed Sobolev space $L^{1,p}(\Omega)$, $1\leq p< \infty$,
is the space of all locally integrable weakly differentiable functions $f:\Omega\to\mathbb{R}$ endowed
with the following seminorm:
\[
\|f\mid L^{1,p}(\Omega)\|=\|\nabla f\mid L^p(\Omega)\|, \,\, 1\leq p<\infty.
\]

The Sobolev space $W^{1,p}_{0}(\Omega)$, $1\leq p< \infty$, is the closure in the $W^{1,p}(\Omega)$-norm of the
space $C^{\infty}_{0}(\Omega)$, where $C^{\infty}_{0}(\Omega)$ is the
space of all infinitely differentiable functions compactly supported in $\Omega$.

The seminormed Sobolev space $L^{1,p}_{0}(\Omega)$, $1\leq p< \infty$, is the closure in the $L^{1,p}(\Omega)$-seminorm of the
space $C^{\infty}_{0}(\Omega)$.

We consider the Sobolev spaces as Banach spaces of equivalence classes of functions up to a set of $p$-capacity zero \cite{M,U93}.

Let $\Omega$ and $\Omega'$ be domains in $\mathbb R^2$. We say that  a diffeomorphism $\varphi :\Omega' \to \Omega$ induces a bounded composition operator
\[
\varphi^*:L^{1,p}(\Omega) \to L^{1,q}(\Omega'),\,\,1\leq q\leq p <\infty,
\]
by the composition rule $\varphi^{*}(f)=f \circ \varphi$, if for any $f \in L^{1,p}(\Omega)$ the composition
$\varphi^{*}(f) \in L^{1,q}(\Omega')$ and there exists a constant $K< \infty$ such that
\[
\|\varphi^{*}(f)\mid L^{1,q}(\Omega')\|\leq K \|f\mid L^{1,p}(\Omega)\|.
\]

The following fact plays a key role in the conformal theory of composition operators on uniform Sobolev spaces (see, for example, \cite{C50,GMU}):

\begin{lemma} \label{L2.2}
Let $\Omega,\Omega'$ be domains in $\mathbb R^2$. Then any conformal mapping $\varphi :\Omega' \to \Omega$  induces, by the composition rule $\varphi^{*}(f)=f \circ \varphi$,
an isometry of Sobolev spaces $L^{1,2}_{0}(\Omega')$ and $L^{1,2}_{0}(\Omega)$:
\[
\|\varphi^{*}(f)\,|\,L^{1,2}_{0}(\Omega')\|=\|f\,|\,L^{1,2}_{0}(\Omega)\|
\]
for any $f \in L^{1,2}_{0}(\Omega)$.
\end{lemma}

\section{Poincar\'e-Sobolev inequalities}

We need the following result from \cite{Pch} in the case of conformal weights.

\begin{theorem}\label{Th3.1}
Let $\Omega \subset \mathbb R^2$ be a simply connected planar domain with non-empty boundary and $h(x,y) =J_{\varphi^{-1}}(x,y)$ is the conformal weight defined by a conformal mapping $\varphi : \Omega' \to \Omega$.
Then for any function $f \in W^{1,2}_{0}(\Omega)$, the weighted Poincar\'e-Sobolev inequality
\[
\left(\iint\limits_\Omega |f(x,y)|^rh(x,y)dxdy\right)^{\frac{1}{r}} \leq A_{r,2}(h,\Omega)
\left(\iint\limits_\Omega |\nabla f(x,y)|^2 dxdy\right)^{\frac{1}{2}}
\]
holds for any $r \geq 2$ with the constant $A_{r,2}(h,\Omega) \leq A_{r,2}({\Omega'})$.
\end{theorem}

Here $A_{r,2}({\Omega'})$ is the best constant in the (non-weight) Poincar\'e-Sobolev
inequality in a bounded domain $\Omega' \subset \mathbb R^2$  with the upper estimate (see, \cite{GPU2019}):

\[
A_{r,2}({\Omega'}) \leq \inf\limits_{p\in \left(\frac{2r}{r+2},2\right)}
\left(\frac{p-1}{2-p}\right)^{\frac{p-1}{p}}
\frac{\left(\sqrt{\pi}\cdot\sqrt[p]{2}\right)^{-1}|{\Omega'}|^{\frac{1}{r}}}{\sqrt{\Gamma(2/p) \Gamma(3-2/p)}}.
\]

The property of the conformal $\alpha$-regularity implies the integrability
of the Jacobian of conformal mappings and therefore for any conformal $\alpha$-regular domain we have the embedding of weighted Lebesgue spaces $L^r(\Omega,h)$ into non-weighted Lebesgue spaces $L^s(\Omega)$ for $s=\frac{\alpha -2}{\alpha}r$ \cite{GU16}.
\begin{lemma}\label{Prop-reg}
Let $\Omega$ be a conformal $\alpha$-regular domain about a domain $\Omega'$. Then for any
function $f \in L^r(\Omega,h)$, $\alpha/(\alpha -2) \leq r < \infty$, the inequality
\[
\|f\mid L^s(\Omega)\|\leq \left(\iint\limits_{\Omega'}|\varphi'(u,v)|^{\alpha}dudv\right)^{\frac{2}{\alpha} \cdot \frac{1}{s}}
\|f\mid L^r(\Omega,h)\|
\]
holds for $s=\frac{\alpha -2}{\alpha}r$.
\end{lemma}

The following theorem gives the upper estimate for the constant in the Poincar\'e-Sobolev inequality as an application of Theorem~\ref{Th3.1} and Lemma~\ref{Prop-reg}:
\begin{theorem}\label{Th4.3}
Let $\Omega$ be a conformal $\alpha$-regular domain about a domain $\Omega'$. Then
\begin{enumerate}
\item for any function $f \in W^{1,2}_{0}(\Omega)$ and for any $s \geq 1$, the Poincar\'e-Sobolev inequality
\[
\|f\mid L^s(\Omega)\| \leq A_{s,2}(\Omega)\|\nabla f\mid L^{2}(\Omega)\|
\]
holds with the constant
$$
A_{s,2}(\Omega) \leq A_{\frac{\alpha s}{\alpha-2},2}(\Omega') \|\varphi'\mid L^{\alpha}(\Omega')\|^{\frac{2}{s}}, \quad 2<\alpha <\infty;
$$
\item if $\alpha = \infty$ then for any function $f \in W^{1,2}_{0}(\Omega)$, the Poincar\'e-Sobolev inequality
\[
\|f\mid L^2(\Omega)\| \leq A_{2,2}(\Omega)\|\nabla f\mid L^{2}(\Omega)\|
\]
holds with the constant
$$A_{2,2}(\Omega) \leq A_{2,2}(\Omega') \big\|\varphi'\mid L^{\infty}(\Omega')\big\|\,.
$$
Here $\varphi:\Omega' \to \Omega$ is the Riemann conformal mapping of $\Omega'$ onto $\Omega$.
\end{enumerate}
\end{theorem}

\begin{remark}
The constant $A_{2,2}^2({\Omega'})=1/\lambda_1({\Omega'})$, where $\lambda_1({\Omega'})$ is the first Dirichlet eigenvalue of the Laplace operator in a domain ${\Omega'}\subset\mathbb R^2$.
\end{remark}

\begin{proof}
1. Let a function $f \in W_0^{1,2}(\Omega)$. Then by Lemma~\ref{Prop-reg} in the case $h(x,y):=J_{\varphi^{-1}}(x,y)=|(\varphi^{-1})'(x,y)|^2$ we get
\begin{multline*}
\left(\iint\limits_{\Omega}|f(x,y)|^sdxdy\right)^\frac{1}{s} \\
\leq
\left(\iint\limits_{\Omega'}|\varphi'(u,v)|^{\alpha}dudv\right)^{\frac{2}{\alpha} \cdot \frac{1}{s}}
\left(\iint\limits_{\Omega}|f(x,y)|^{r}|(\varphi^{-1})'(x,y)|^2dxdy\right)^{\frac{1}{r}}
\end{multline*}
for $s=\frac{\alpha -2}{\alpha}r$.

According to Theorem~\ref{Th3.1} we have
\[
\left(\iint\limits_\Omega |f(x,y)|^r|(\varphi^{-1})'(x,y)|^2dxdy\right)^{\frac{1}{r}} \leq A_{r,2}(\Omega')
\left(\iint\limits_\Omega |\nabla f(x,y)|^2 dxdy\right)^{\frac{1}{2}}.
\]

Combining these inequalities and given that $r=\alpha s/(\alpha -2)$ we obtain
\[
\left(\iint\limits_{\Omega}|f(x,y)|^sdxdy\right)^\frac{1}{s} \\
\leq
A_{\frac{\alpha s}{\alpha-2},2}(\Omega') \|\varphi'\mid L^{\alpha}(\Omega')\|^{\frac{2}{s}}
\left(\iint\limits_\Omega |\nabla f(x,y)|^2 dxdy\right)^{\frac{1}{2}}
\]
for any $s \geq 1$.

2. Let a function  $f\in W_0^{1,2}(\Omega)$. Given the following conformal equality
$$
|\varphi'(u,v))|^{-2}=|(\varphi^{-1})'(x,y)|^2
$$
holds for all $(u,v)\in \Omega'$ and for all $(x,y)\in \Omega$ we obtain
\begin{multline*}
\left(\iint\limits_{\Omega} |f(x,y)|^2~dxdy\right)^{\frac{1}{2}}
=\left(\iint\limits_{\Omega} |f(x,y)|^2|(\varphi^{-1})'(x,y)|^{-2}|(\varphi^{-1})'(x,y)|^2~dxdy\right)^{\frac{1}{2}} \\
\leq \|(\varphi^{-1})' \mid L^{\infty}(\Omega)\|^{-1} \left(\iint\limits_{\Omega} |f(x,y)|^2||(\varphi^{-1})'(x,y)|^2~dxdy\right)^{\frac{1}{2}} \\
= \|\varphi' \mid L^{\infty}(\Omega')\|
\left(\iint\limits_{\Omega} |f(x,y)|^2||(\varphi^{-1})'(x,y)|^2~dxdy\right)^{\frac{1}{2}}.
\end{multline*}
Applying the change of variable formula for conformal mappings we have
\[
\left(\iint\limits_{\Omega} |f(x,y)|^2||(\varphi^{-1})'(x,y)|^2~dxdy\right)^{\frac{1}{2}} =
\left(\iint\limits_{\Omega'} |f \circ \varphi(u,v)|^2~dudv\right)^{\frac{1}{2}}.
\]
Given (non-weighed) Poincar\'e-Sobolev inequality \cite{M} we get
\[
\left(\iint\limits_{\Omega'} |f \circ \varphi(u,v)|^2~dudv\right)^{\frac{1}{2}}
\leq
A_{2,2}(\Omega')
\left(\iint\limits_{\Omega'} |\nabla (f \circ \varphi(u,v))|^2~dudv\right)^{\frac{1}{2}}.
\]
By Lemma~\ref{L2.2}
\[
\left(\iint\limits_{\Omega'} |\nabla (f \circ \varphi(u,v))|^2~dudv\right)^{\frac{1}{2}} =
\left(\iint\limits_\Omega |\nabla f(x,y)|^2~dxdy\right)^{\frac{1}{2}}.
\]
Thus, by combining the expressions obtained above, we finally obtain the desired result:
\[
\left(\iint\limits_{\Omega} |f(x,y)|^2~dxdy\right)^{\frac{1}{2}} \\
\leq \|\varphi' \mid L^{\infty}(\Omega')\|
A_{2,2}(\Omega')
\left(\iint\limits_\Omega |\nabla f(x,y)|^2~dxdy\right)^{\frac{1}{2}}.
\]
\end{proof}

\section{Lower estimates for $\lambda_1(\Omega)$}

We consider the Dirichlet eigenvalue problem~\eqref{Dir} in the weak formulation:
\[
\iint\limits_{\Omega} \left\langle \nabla u(x,y), \nabla v(x,y) \right\rangle dxdy
= \lambda \iint\limits_{\Omega} u(x,y)v(x,y)~dxdy, \quad v \in W_0^{1,2}(\Omega).
\]

Recall that the first Dirichlet eigenvalue of the Laplace operator is defined by
\[
\lambda_1(\Omega)= \inf_{u \in W_0^{1,2}(\Omega) \setminus \{0\}} \frac{\iint\limits_{\Omega} |\nabla u(x,y)|^2~dxdy}{\iint\limits_{\Omega} |u(x,y)|^2~dxdy}\,.
\]
In other words, $\lambda_1^{-\frac{1}{2}}(\Omega)$ is the exact constant $A_{2,2}(\Omega)$ in the Poincar\'e-Sobolev inequality
\[
\left( \iint\limits_{\Omega} |u(x,y)|^2~dxdy \right)^{\frac{1}{2}} \leq A_{2,2}(\Omega)
\left( \iint\limits_{\Omega} |\nabla u(x,y)|^2~dxdy \right)^{\frac{1}{2}}, \quad u \in W_0^{1,2}(\Omega).
\]

\begin{theorem}\label{Th4.1}
Let $\Omega$ be a conformal $\alpha$-regular domain about a domain $\Omega'$. Then
\[
\frac{1}{\lambda_1(\Omega)} \leq A^2_{\frac{2\alpha}{\alpha-2},2}(\Omega') \|\varphi'\mid L^{\alpha}(\Omega')\|^2,
\]
where $\varphi:\Omega' \to \Omega$ is the Riemann conformal mapping of $\Omega'$ onto $\Omega$ and
\[
A_{\frac{2\alpha}{\alpha -2},2}(\Omega') \leq \inf\limits_{p\in \left(\frac{\alpha}{\alpha -1},2\right)}
\left(\frac{p-1}{2-p}\right)^{\frac{p-1}{p}}
\frac{\left(\sqrt{\pi}\cdot\sqrt[p]{2}\right)^{-1}|\Omega'|^{\frac{\alpha-2}{2\alpha}}}{\sqrt{\Gamma(2/p) \Gamma(3-2/p)}}.
\]
\end{theorem}

\begin{proof}
By the minimax principle and Theorem~\ref{Th4.3} in the case $s=2$, we have
\[
\iint\limits_{\Omega}|u(x,y)|^2~dxdy \leq A^2_{2,2}(\Omega)
\iint\limits_{\Omega} |\nabla u(x,y)|^2~dxdy,
\]
where
\[
A_{2,2}(\Omega)\leq A_{\frac{2\alpha}{\alpha-2},2}(\Omega') \|\varphi'\mid L^{\alpha}(\Omega')\|.
\]
Thus
\[
\frac{1}{\lambda_1(A,\Omega)} \leq A^2_{\frac{2\alpha}{\alpha-2},2}(\Omega') \|\varphi'\mid L^{\alpha}(\Omega')\|^2.
\]
\end{proof}

This theorem can be reformulated in terms of the domain conformal radius \cite{Avkh,BM}:

\begin{theorem}\label{th4.2}
Let $\Omega$ be a conformal $\alpha$-regular domain about a domain $\Omega'$. Then
\[
\frac{1}{\lambda_1(\Omega)} \leq A^2_{\frac{2\alpha}{\alpha-2},2}(\Omega')
\left(\iint\limits_{\Omega'} \left( \frac{R_{\Omega}(\varphi(u,v))}{R_{\Omega'}(u,v)}\right)^{\alpha}dudv\right)^{\frac{2}{\alpha}},
\]
where $\varphi:\Omega' \to \Omega$ is the Riemann conformal mapping of $\Omega'$ onto $\Omega$,
$R_{\Omega'}(u,v)$ and $R_{\Omega}(\varphi(u,v))$ are conformal radii of $\Omega'$ and $\Omega$,
\[
A_{\frac{2\alpha}{\alpha -2},2}(\Omega') \leq \inf\limits_{p\in \left(\frac{\alpha}{\alpha -1},2\right)}
\left(\frac{p-1}{2-p}\right)^{\frac{p-1}{p}}
\frac{\left(\sqrt{\pi}\cdot\sqrt[p]{2}\right)^{-1}|\Omega'|^{\frac{\alpha-2}{2\alpha}}}{\sqrt{\Gamma(2/p) \Gamma(3-2/p)}}.
\]
\end{theorem}

\begin{remark}
In the case of the unit disc $\mathbb D$, the conformal radius is defined by the equality $R_{\mathbb D}(w)=1-|w|^2$, $w=(u,v)$.
\end{remark}

In the limit case $\alpha=\infty$ we have the following assertion:

\vskip 0.2cm
\noindent
{\bf Theorem~A.}
\textit{Let $\Omega$ be a conformal $\infty$-regular domain about a domain $\Omega'$. Then
\begin{equation}\label{Estimate}
\lambda_1(\Omega) \geq \frac{\lambda_1(\Omega')}{\big\|\varphi' \mid L^{\infty}(\Omega')\big\|^2},
\end{equation}
where $\varphi:\Omega' \to \Omega$ is the Riemann conformal mapping of $\Omega'$ onto $\Omega$.
}

\vskip 0.2cm

As an application of Theorem~A we consider the following examples.

\begin{example}
\label{example1}
The diffeomorphism
$$
\varphi(z)=e^z, \quad z=x+iy,
$$
is conformal and maps the rectangle $Q:=(0,d) \times (0,\pi)$
in the $z$-plane onto the interior of the domain
\[
\Omega:= \left\{w:1<|w|<e^d,\,\, 0<\arg w<\pi \right\}
\]
in the $w$-plane.

Now we estimate the norm of the derivative $\varphi'$ in $\big\|\varphi' \mid L^{\infty}(Q)\big\|$.
Given that, $|\varphi'(z)|^2=e^{2x}$, we obtain
\[
\big\|\varphi' \mid L^{\infty}(\Omega')\big\|^2=\esssup\limits_{Q}e^{2x}\leq e^{2d}.
\]

Since $\lambda_1(Q)=1+\pi^2/d^2$, by Theorem~A we have
\begin{equation}\label{K-P}
\lambda_1(\Omega_e)\geq \frac{\pi^2+d^2}{d^2e^{2d}}.
\end{equation}

Applying the Rayleigh-Faber-Krahn and Makai ($\gamma =1/4$, $\rho =(e^d-1)/2$) inequalities for the domain $\Omega$, we obtain, respectively
$$
\lambda_1(\Omega_e) \geq \frac{2(j_{0,1})^2}{e^{2d}-1} \quad \text{and} \quad \lambda_1(\Omega_e) \geq \frac{1}{(e^d-1)^2}.
$$

For example, if $0<d\leq \ln 2$, then the following inequalities
\[
\frac{\pi^2+d^2}{d^2e^{2d}}>\frac{2(j_{0,1})^2}{e^{2d}-1}>\frac{1}{(e^d-1)^2}.
\]
hold. This means that for $0<d\leq \ln 2$, the estimate~\eqref{K-P} is better than the estimates based on the Rayleigh-Faber-Krahn and Makai inequalities. Several estimates of $\lambda_1(\Omega)$ are given in Table 1.

\begin{table}[h]
\caption{Estimates of $\lambda_1(\Omega)$}
\begin{center}
\begin{tabular}{ccccc}
d: & 1 & $\ln 2$ & $\ln \sqrt{3}$ & $\ln \sqrt{2}$ \\
Makai: & 0,338 & 1 & 1,866 & 5,828 \\
RFK: & 1,810 & 3,855 & 5,783 & 11,566 \\
Estimate~\eqref{K-P}: & 1,471 & 5,385 & 11,236 & 41,584 \\
\end{tabular}
\end{center}
\end{table}
\end{example}

\begin{example}
\label{example1}
The diffeomorphism
$$
\varphi(z)=\sin z, \quad z=x+iy,
$$
is conformal and maps the rectangle $Q:=(- \pi /2,\pi /2) \times ([-d,d)$
in the $z$-plane onto the interior of the ellipse
\[
\Omega_e:= \left(\frac{u}{\cosh d}\right)^2+\left(\frac{v}{\sinh d}\right)^2=1
\]
in the $w$-plane with slits from the foci $(\pm 1,0)$ to the tips of the major semi-axes $(\pm \cosh d,0)$.

Now we estimate the norm of the derivative $\varphi'$ in $\big\|\varphi' \mid L^{\infty}(Q)\big\|$.
Given that, $|\varphi'(z)|^2=(\cos 2x+\cosh 2y)/2$, we obtain
\[
\big\|\varphi' \mid L^{\infty}(\Omega')\big\|^2=\esssup\limits_{Q}\frac{1}{2}(\cos 2x+\cosh 2y)
\leq \frac{1}{2}(1+\cosh 2d)=\cosh^{2}d.
\]

Since $\lambda_1(Q)=1+\pi^2/(2d)^2$, by Theorem~A we have
\begin{equation}\label{KP}
\lambda_1(\Omega_e)\geq \frac{\pi^2+(2d)^2}{(2d\cosh d)^2}.
\end{equation}

Applying the Rayleigh-Faber-Krahn and Makai ($\gamma =1/4$, $\rho =d$) inequalities for the domain $\Omega_e$, we obtain, respectively
$$
\lambda_1(\Omega_e) \geq \frac{2(j_{0,1})^2}{\sinh 2d} \quad \text{and} \quad \lambda_1(\Omega_e) \geq \frac{1}{4d^2}.
$$

For example, if $0<d\leq1/3$, then the following inequalities
\[
\frac{\pi^2+(2d)^2}{(2d\cosh d)^2}>\frac{2(j_{0,1})^2}{\sinh 2d}>\frac{1}{4d^2}.
\]
hold. This means that for $0<d\leq1/3$, the estimate~\eqref{KP} is better than the estimates based on the Rayleigh-Faber-Krahn and Makai inequalities. Several estimates of $\lambda_1(\Omega_e)$ are given in Table 2.

\begin{table}[h]
\caption{Estimates of $\lambda_1(\Omega_e)$}
\begin{center}
\begin{tabular}{ccccc}
d: & 1/2 & 1/3 & 1/4 & 1/8 \\
Makai: & 1 & 1,25 & 4 & 16 \\
RFK: & 9,841 & 16,127 & 22,195 & 45,786 \\
Estimate~\eqref{KP}: & 8,548 & 20,807 & 38,050 & 156,456 \\
\end{tabular}
\end{center}
\end{table}
\end{example}

Using Theorem~A and the domain monotonicity property of the Dirichlet eigenvalues for the Laplace operator (see, for example, \cite{Henr,M}), we obtain the following result.
\begin{theorem}\label{var}
Let $\Omega$ be a conformal $\infty$-regular domain about a domain $\Omega'$. We assume that $\Omega \subset \Omega'$. Then
\begin{equation*}
\lambda_1(\Omega)-\lambda_1(\Omega') \geq \frac{1-\big\|\varphi' \mid L^{\infty}(\Omega')\big\|^2}{\big\|\varphi' \mid L^{\infty}(\Omega')\big\|^2} \lambda_1(\Omega'),
\end{equation*}
where $\varphi:\Omega' \to \Omega$ is the Riemann conformal mapping of $\Omega'$ onto $\Omega$.
\end{theorem}

\begin{proof}
Because $\Omega \subset \Omega'$ we have $\lambda_1(\Omega) \geq \lambda_1(\Omega')$. Taking into account the inequality~\eqref{Estimate} in Theorem~A and making some calculation, we find
\begin{equation*}
\lambda_1(\Omega)-\lambda_1(\Omega') \geq \frac{1-\big\|\varphi' \mid L^{\infty}(\Omega')\big\|^2}{\big\|\varphi' \mid L^{\infty}(\Omega')\big\|^2} \lambda_1(\Omega').
\end{equation*}
The theorem is proved.
\end{proof}

\subsection{Estimates in convex domains}
We precise Theorem~A and estimates of the spectral gap \cite{GPU18} using an explicit upper bound for the derivative of a
conformal mapping of the unit disc onto a convex domain from \cite{Kov}.
\begin{definition}\label{def}
Suppose $\Omega$ is a simply connected convex domain that contains 0 and has $C^{1,1}$-smooth boundary. We say that such a domain satisfies the $(R_O,R_I,R_C)$ condition if:

$\bullet$ $R_O$, $R_I$, $R_C$ are all positive,

$\bullet$ $R_O$ is the minimal $r$ such that $\Omega \subset D(0, r)$,

$\bullet$ $R_I$ is the maximum $r$ such that $D(0, r) \subset \Omega$,

$\bullet$ $\Omega$ can be expressed as a union of open discs of radius $R_C$.
\end{definition}

The subscripts in Definition~\ref{def} serve to indicate that $R_O$ is the outer radius, $R_I$ the
inner radius, and $R_C$ a curvature radius.
\begin{theorem}\label{Kov}
Let $\Omega$ satisfy the $(R_O,R_I,R_C)$ condition in Definition~\ref{def}. Then for any
conformal mapping $\varphi:\mathbb D \to \Omega$ fixing 0 we have
\begin{equation}\label{distortion}
\sup |\varphi'| \leq R_C \exp \left\{2(R_O-R_C) \frac{\log R_I - \log R_C}{R_I-R_C}\right\}.
\end{equation}
When $R_I=R_C$, the difference quotient is understood as $1/R_I$.
\end{theorem}

Given the estimate~\eqref{distortion} in Theorem~\ref{Kov} we obtain the lower bound for the first Dirichlet eigenvalue of the Laplace operator for a convex domain.
\begin{theorem}\label{convex}
Let $\Omega$ be a convex conformal $\infty$-regular domain about the unit disc $\mathbb D$ and satisfy the $(R_O,R_I,R_C)$ condition in Definition~\ref{def}. Then
\[
\lambda_1(\Omega) \geq \frac{j_{0,1}^2}{R_C^2} \exp \left\{-4(R_O-R_C) \frac{\log R_I - \log R_C}{R_I-R_C}\right\},
\]
where $j_{0,1} \approx 2.4048$ is the first positive zero of the Bessel function $J_0$.
\end{theorem}

In the work \cite{GPU18} were obtained asymptotically sharp lower estimates
for the spectral gap between the first two Dirichlet eigenvalues for conformal regular domains.
\begin{theorem}\label{gap}
Let $\Omega$ be a conformal $\infty$-regular domain about the unit disc $\mathbb D$ of area $\pi$.
Then
\begin{multline*}
\lambda_2(\Omega)-\lambda_1(\Omega) \\
\geq \lambda_2(\mathbb D)-\lambda_1(\mathbb{D})-(\lambda_{*}^2 +1)\lambda_1^2(\mathbb D_{\rho}) \gamma_{\infty}  \bigl(\|\varphi' \mid L^{\infty}(\mathbb D)\| + 1 \bigr) \|\varphi'-1 \mid L^{2}(\mathbb D)\|,
\end{multline*}
where $\lambda_{*} \approx 2.539$, $\mathbb D_{\rho}$ is the largest disc inscribed in $\Omega$ and
\[
\gamma_{\infty} = \inf\limits_{p\in \left(\frac{4}{3},2\right)}
\left(\frac{p-1}{2-p}\right)^{\frac{2(p-1)}{p}}
\frac{\pi^{-\frac{1}{2}} 4^{-\frac{1}{p}}}{\Gamma(2/p) \Gamma(3-2/p)}.
\]
\end{theorem}

In accordance with estimate~\eqref{distortion}, Theorem~\ref{gap} is rewritten as
\begin{theorem}
Let $\Omega$ be a convex conformal $\infty$-regular domain about the unit disc $\mathbb D$ of area $\pi$ and satisfy the $(R_O,R_I,R_C)$ condition in Definition~\ref{def}. Then
\begin{multline*}
\lambda_2(\Omega)-\lambda_1(\Omega)
\geq \lambda_2(\mathbb D)-\lambda_1(\mathbb{D})
-(\lambda_{*}^2 +1)\lambda_1^2(\mathbb D_{\rho}) \gamma_{\infty}  \\
\times \left(R_C \exp \left\{2(R_O-R_C) \frac{\log R_I - \log R_C}{R_I-R_C}\right\}+1\right)
\|\varphi'-1 \mid L^{2}(\mathbb D)\|,
\end{multline*}
where $\lambda_{*} \approx 2.539$, $\mathbb D_{\rho}$ is the largest disc inscribed in $\Omega$ and
\[
\gamma_{\infty} = \inf\limits_{p\in \left(\frac{4}{3},2\right)}
\left(\frac{p-1}{2-p}\right)^{\frac{2(p-1)}{p}}
\frac{\pi^{-\frac{1}{2}} 4^{-\frac{1}{p}}}{\Gamma(2/p) \Gamma(3-2/p)}.
\]
\end{theorem}

\vskip 0.3cm

\textbf{Acknowledgements.}
The authors thank Vladimir Gol'dshtein and Alexander Ukhlov for useful discussions and valuable comments.
This work was supported by RSF Grant No. 23-21-00080.

\vskip 0.3cm

Regional Scientific and Educational Mathematical Center, Tomsk State University, 634050 Tomsk, Lenin Ave. 36, Russia
							
\emph{E-mail address:} \email{ia.kolesnikov@mail.ru}

Regional Scientific and Educational Mathematical Center, Tomsk State University, 634050 Tomsk, Lenin Ave. 36, Russia
							
\emph{E-mail address:} \email{va-pchelintsev@yandex.ru} \\

\end{document}